\documentclass[10pt,a4paper,reqno]{amsproc}
\usepackage{longtable,cite}
\usepackage{lscape}
\usepackage{multirow}
\usepackage[pdftex,bookmarks=true]{hyperref}
\newtheorem{theorem}{Theorem}[section]
\newtheorem{lemma}[theorem]{Lemma}

\theoremstyle{definition}
\newtheorem{definition}[theorem]{Definition}

\theoremstyle{remark}

\numberwithin{equation}{section}
\usepackage{tikz}
\usetikzlibrary{shapes, arrows}
\usetikzlibrary{chains,arrows.meta}
\begin{document}
	
	\title [New convolution related theorems and applications associated with OLCT]{New convolution related theorems and applications associated with offset linear canonical transform} 
	\author{Gita Rani Mahato}
	\author{Sarga Varghese}
	\author{Manab Kundu}
	
	\address{Department of Mathematics, SRM University AP, Amaravati-522240, India}

	\email{\hfill \break
	gitamahato1158$@$gmail.com (Gita Rani Mahato),
	\hfill \break sargavarghese$@$gmail.com (Sarga Varghese),
	\hfill \break manabiitism17$@$gmail.com (Manab Kundu-Corresponding author))}
	
	\thanks{Corresponding author:Manab Kundu}
	\subjclass[2010]{43A32, 44A35, 42A85}
	\date{}

	\keywords{ Fourier transform; Offset linear canonical transform; Convolution; Correlation}
	
	\begin{abstract}
	In this paper, we define new type of convolution and correlation theorems associated with the offset linear canonical transform (OLCT). Additionally, we discuss their applications in multiplicative filter design, which may prove useful in optics and signal processing for signal recovery. Furthermore, we explore the real Paley-Wiener (PW) and Boas theorems for the OLCT, analyzing signal characteristics for OLCT within the $L^2$ domain.
	\end{abstract}

	\maketitle
	
	
	\section{Introduction}
The Fourier transform (FT), introduced by Joseph Fourier in the early 19th century, is a widely used technique in signal processing and has been extensively applied across various scientific fields. Several generalizations of the Fourier transform, such as the linear canonical transform (LCT), fractional Fourier transform (FrFT), and offset linear canonical transform (OLCT), provide a greater degree of parameterization compared to the Fourier transform \cite{bdr,alm, gos}.

The offset linear canonical transform (OLCT) \cite{nw2}, a generalization of the LCT, is a time-shifted and frequency-modulated version of the LCT. It is also known as the special affine Fourier transform or the inhomogeneous canonical transform. The OLCT encompasses several important signal processing operators and is widely used in optical system modeling. Many operations, including the time-shifting, scaling, and frequency modulation of FT, FrFT, Fresnel transform (FRST), offset FrFT, LCT are special cases of the OLCT. As the theories of the FrFT and LCT have advanced, the OLCT has become a vital tool in numerous engineering fields. In certain signal processing applications, where the FT, FrFT, or LCT may not be directly applicable, the OLCT proves to be an essential alternative. Therefore, exploring and further developing the theory of the OLCT is both valuable and important \cite{xq, huo, dw2}.

Considering the developments in the study of the offset linear canonical transform (OLCT), several remarkable contributions have been made by many researchers. Verma et al. \cite{verma} established the offset linear canonical wavelet transform (OLCWT) as a continuous linear operator, enhancing the understanding of wavelet transforms in mathematical analysis. Bhat et al. \cite{bhat1, bhat2}  explored the Wigner-Ville distribution (WVD) associated with the quaternion offset linear canonical transform (QOLCT), introducing novel convolution and correlation operators, and emphasizing their role in signal processing with reconstruction formulas and orthogonality relations. Ahmad et al. \cite{ahmad} introduced the WVD based on the biquaternion offset linear canonical transform (WVD-BiQOLCT), extending uncertainty principles and exploring its applications in signal processing and quantum mechanics. Minh \cite{minh} developed the offset fractional Fourier Wigner distribution (OFrWD) and the offset fractional Fourier ambiguity function (OFrAF), deriving key properties and exploring their use in designing multiplicative filters and detecting parameters of linear frequency-modulated signals in the offset fractional Fourier transform domain.

In applied mathematics, two of the most prominent spectrum-related theorems are the Boas and Paley-Wiener (PW) theorems. In 1934, Paley and Wiener \cite{pw12} published the initial version of the PW theorem. This theorem states that a function $f \in L^2 (\mathbb{R}^n)$ is the Fourier transform of a square-integrable function with compact support if and only if it is the restriction to $\mathbb{R}^n$ of an entire function of exponential type. Recently, there has been significant interest in real versions of the PW and Boas theorems, discovered by Bang and Tuan \cite{bang, tuan, pw1, pw2, pw8}, respectively. These theorems characterize band-limited and high-pass signals using derivative and integral operators. Numerous studies have focused on extending these theories to higher dimensions and exploring their applications with other integral transforms \cite{pw3, pw4, pw5, pw6, pw7, pw9, pw10, pw11}.
	
Motivated by the characteristics of convolution, correlation, and other properties of the OLCT, we have introduced a novel approach to convolution and correlation associated with the OLCT. Additionally, we have investigated various spectrum-related principles connected to the OLCT. Furthermore, several applications of the OLCT have been explored.
	
The organization of this work is as follows: Section 2 provides a brief overview of the fundamental concepts of the FT, OLCT, convolution, and correlation. In Section 3, we establish new theorems related to convolution and correlation in the context of the OLCT. Section 4 focuses on spectral theories associated with the OLCT, while Section 5 explores potential applications. Finally, Section 6 presents the concluding remarks.

	
	\section{Preliminaries} 
	This section deals with some elementary concepts and notations that will be helpful in the current study.
	
	\begin{definition}
		Let the function $f \in L^1(\mathbb{R})$. Then we can  define the FT of $f$ as 
		\begin{eqnarray}
			({F}f)(u) = \frac{1}{\sqrt{2\pi}} \int_{\mathbb{R}} e^{-ixu}f(x)dx,\hspace{3mm} \forall u\in \mathbb{R}.
		\end{eqnarray}
		inverse Fourier transform can be written as
		\begin{eqnarray}
			f(x)=\frac{1}{\sqrt{2\pi}} \int_{\mathbb{R}} e^{ixu}({F}f)(u)dx,\hspace{3mm} \forall x\in \mathbb{R}.
		\end{eqnarray}
	\end{definition}
	\begin{definition}
		The LCT of $f \in L^1(\mathbb{R})$ with the parameters in terms of a unimodular matrix $M = 
		\begin{pmatrix}
			a & b\\
			c & d
		\end{pmatrix}, (ad-bc=1)$ is defined as 
		\begin{eqnarray}
		\small
		\mathcal{L}_M (u)=	
		\begin{cases}
		\sqrt{\frac{1}{2\pi i b}}	\int_{\mathbb{R}}f(x)~e^{\frac{i}{2b}(a x^2 - 2xu  + du^2  )}dx & \hspace{-4mm}, ~~ b\neq0\\
		\frac{1}{\sqrt{a}}	~e^{i\frac{c}{2a}(u)^2}\delta(x-\frac{u}{a})& \hspace{-0.3cm}, ~~b=0, 
		\end{cases}
	\end{eqnarray}
		
	\end{definition}
	
	\begin{definition}\label{def1}
		For a given set of parameters $M=\{ a,b,c,d,u_0,\omega_0\}$, $b\neq 0$, the offset linear canonical  transform (OLCT) of a function $f \in L^1(\mathbb{R}) \cap L^2(\mathbb{R})$ can be defined as\cite{abe} 
		\begin{eqnarray}
			\small
			F_M (u)= O^M[f(x)](u)=
			\begin{cases}
				\int_{\mathbb{R}}f(x)\mathcal{K}_M(x,u)dx & \hspace{-4mm}, ~~ b\neq0\\
				\sqrt{d}~e^{i\frac{cd}{2}(u-u_0)^2+i\omega_0u}f[d(u-u_0)] & \hspace{-0.3cm}, ~~b=0, 
			\end{cases}
		\end{eqnarray}
		\begin{eqnarray}
			\mathcal{O}^{M}\big[f(t)\big] (u)=  {F_M}(u)=  \int_{\mathbb{R}} \mathcal{K}_{M}(u,x) f(x) dx, \hspace{3mm} \omega\in \mathbb{R},
		\end{eqnarray}
		where
		\begin{eqnarray}
			\mathcal{K}_{M} (u, x)=  \sqrt{\frac{1}{2\pi i b}}~e^{\frac{i}{2b}(du_0^2)}~e^{\frac{i}{2b}(a x^2 + 2  x (u_0-u) - 2 u(du_0-b\omega_0) + du^2  )}
		\end{eqnarray}
		is the kernel of OLCT and	$ a, b, c, d, u_0, \omega_0 \in \mathbb{R}$.
		Then its  inversion transformation can be defined as 
		\begin{eqnarray}
			f(x)=  \Big(\mathcal{O}^{M^{-1}} (F_{M} f)\Big)(x)= C \int_{\mathbb{R}}  \mathcal{K}_{M^{-1}} (x, u) (\mathcal{O}^{M} f)(u) du,
		\end{eqnarray}
		where parameter 
		\begin{eqnarray*}
			M^{-1}=(d,-b,-c,a,b\omega_0-bu_0)  
			and \\ C= e^{\frac{i}{b}(cdu_0^2-2adu_0\omega_0+ab\omega_0^2)}.
		\end{eqnarray*}
	\end{definition}
	
	\begin{lemma}(Reimann-Lebesgue)\label{lem1} 
		If $f \in L^1(\mathbb{R})$, then $\mathcal{O}^{M}(u)  \in \mathcal{C}_0(\mathbb{R})$.
		\\i.e., 
		\begin{eqnarray}
			(\mathcal{O}^{M} f)(u) \to 0 \;\;\;\; as \;\;\;\; |u| \to \infty.
		\end{eqnarray}		
	\end{lemma}
	\begin{proof}
		let $f\in L^1(\mathbb{R})$
		\begin{eqnarray*}
			|O^M{f}(u)|&=& \bigg|\sqrt{\frac{1}{{2\pi i b}}}e^{\frac{i}{2b}du^2}\int_{\mathbb{R}}e^{\frac{i}{2b}(ax^2+2x(u_0-u)-2u(du_0-b\omega_0)+du^2)f(t)dt}\bigg|
			\\&\leq&\bigg|\sqrt{\frac{1}{2\pi i b}}\bigg|\bigg| e^{\frac{i}{2b}du_0^2-2u(du_0-b\omega_0)+du^2}\bigg|\bigg|\int_{\mathbb{R}}e^{\frac{i}{2b}(ax^2+2x(u_0-u))}f(x)dx\bigg|
			\\&\leq& \sqrt{\frac{1}{2\pi b}}\bigg|\int_{\mathbb{R}}e^{\frac{i}{2b}(ax^2+2x(u_0-u))}f(x)dx\bigg|
			\\&=&\sqrt{\frac{1}{2\pi b}}\bigg|\int_{\mathbb{R}}e^{\frac{-i}{b}ux}g(x)d(x)\bigg|, \text{ where } g(x)=e^{\frac{i}{2b}(ax^2+2xu_0)}
			\\&\to& 0 \text{ as } u \to \infty,
		\end{eqnarray*}
		by classic Riemann Lebesgue lemma for Fourier series.
	\end{proof}
	
	\begin{theorem}\label{nm}
		The inner product of OLCT of $f,g \in L^2 (\mathbb{R})$ is same as the inner product of $f$ and $g$. 
		i.e.,
		\begin{eqnarray*}
			\big< f, g \big> &=& \big< \mathcal{O}^{M}f , \mathcal{O}^{M}g\big>.
		\end{eqnarray*}	
		If we take $f=g$, the equation takes the particular form as
		\begin{eqnarray}
			||f||_2^2 &=& ||\mathcal{O}^{M} f||_2^2. 
		\end{eqnarray}
	\end{theorem}
	
	\begin{definition}
		For the functions $f, g \in L^1 (\mathbb{R})$, we can define convolution by 
		\begin{eqnarray}
			(f * g)(x) = f(x) * g(x) = \int_{\mathbb{R}} f(\tau) g(x-\tau) d\tau,
		\end{eqnarray}
		which satisfies
		\begin{eqnarray}
			\hat{(f * g)}(u) = \sqrt{2 \pi} (\hat{ f})(u) (\hat{ g})(u),
		\end{eqnarray}
		where $\hat{ f}$ denotes the Fourier transform of $f$.
	\end{definition}


	\section{Convolution and correlation}
	
	In this Section, we discuss the generalized version of convolution and correlation  \cite{zc, cas} for the OLCT which are generalized form of convolution and correlation associated with FT.	
	\subsection{OLCT convolution operation}
	\begin{definition}\label{def2}
		Let $f(x), g(x) \in L^1({\mathbb{R}})$. Then the convolution for OLCT can be expressed as follows
		\begin{eqnarray}
			(f \bigoplus g)(x) = \sqrt{\frac{1}{2\pi i b}}~e^{\frac{i}{2b} du_0^2} \int_{\mathbb{R}} f(\tau) g(\frac{x}{2}-\tau) e^{{\frac{-i}{2 b}}[a\{(\frac{x}{2}  +2 \tau)^2 -6\tau^2  +\frac{ x^2}{2}\} +xu_0]} d\tau.\label{1}
		\end{eqnarray}
	\end{definition}
	
	\begin{theorem}
		Let the functions $f(x), g(x) \in L^1({\mathbb{R}})$. Then 
		\begin{eqnarray*}
			f \bigoplus g \in L^1(\mathbb{R}) ~~~~and~~~~ ||f \bigoplus g||_1 \leq \sqrt{\frac{4}{2\pi b }} ||f||_1 ||g||_1.
		\end{eqnarray*}
	\end{theorem}
	\begin{proof}
		Using the definition \ref{def1} and Fubini's theorem, we get
		\begin{eqnarray*}
			&&\int_{\mathbb{R}}|f \bigoplus g(x)|dx 
			\\&&= \int_{\mathbb{R}}\bigg{|}\sqrt{\frac{1}{ 2\pi b i}}~e^{\frac{i}{2b}du_0^2} \int_{\mathbb{R}} f(\tau) g(\frac{x}{2}-\tau) e^{\frac{-i}{2b}[\{(\frac{1}{2} x +2 \tau)^2 -6\tau^2 + \frac{x^2}{2}\} +xu_0]} d\tau \bigg{|}dx
			\\&&\leq \sqrt{\frac{1}{ 2\pi b }} \int_{\mathbb{R}} \int_{\mathbb{R}} \bigg{|}f(\tau) g(\frac{x}{2}-\tau) e^{\frac{-i}{2b}[a\{(\frac{1}{2} x +2 \tau)^2 -6\tau^2 + \frac{x^2}{2}\} +xu_0]}\bigg{|} d\tau dx
			\\&&= \sqrt{\frac{1}{ 2\pi b }} \int_{\mathbb{R}} \big{|}f(\tau) \big{|} \bigg[ \int_{\mathbb{R}} \bigg{|} g(\frac{x}{2}-\tau)\bigg{|}dx \bigg] d\tau 
			\\&&= \sqrt{\frac{4}{2\pi b }} ||f||_1 ||g||_1. 
		\end{eqnarray*}
		This completes the theorem.	
	\end{proof}
	
	\begin{theorem}
		Let $\mathcal{O}^{M} f$, $\mathcal{O}^{M} g$ be the OLCT of $f(x)$ and $g(x)$ respectively, where $f(x), g(x) \in L^2(\mathbb{R})$. Then 
		\begin{eqnarray}
			\mathcal{O}^{M} (f \bigoplus g) (u) =2\mathcal{T}(u) \mathcal{O}^{M} f (2u) \mathcal{O}^{M} g(2u), \label{2}
		\end{eqnarray}
		where
		\begin{eqnarray}
			\mathcal{T}(u)=e^{\frac{-i}{2b}(6u\{du_0-b\omega_0\} + 7du^2)}.
		\end{eqnarray}
	\end{theorem}
	\begin{proof}
		By using the definitions \ref{def1} and \ref{def2},
		\begin{eqnarray*}
			&&	\mathcal{O}^{M} (f\bigoplus g)(u) \\&=&\hspace{-3mm} \sqrt{\frac{1}{2\pi i b}} \int_{\mathbb{R}} \mathcal{K}_{M}(u,x) (f\bigoplus g) (x) dx 
			\\&=& \hspace{-3mm} \sqrt{\frac{1}{2\pi i b}}~e^{\frac{i}{2b}du_0^2}\int_{\mathbb{R}} e^{\frac{i}{2b}(ax^2+2x(u_0-u)-2u(du_0-b\omega_0)+du^2)} \\&&\times \sqrt{\frac{1}{2\pi ib}}~e^{\frac{i}{2b} du_0^2} \int_{\mathbb{R}} f(\tau) g(\frac{x}{2}-\tau) e^{{\frac{-i}{2 b}}[a\{(\frac{x}{2}  +2 \tau)^2 -6\tau^2  +\frac{ x^2}{2} \} +xu_0]} d\tau.dx
		\end{eqnarray*}
		Changing the variable $\frac{x}{2}-\tau = \xi$, we get
		\begin{eqnarray*}
			&&	\mathcal{O}^{M} (f\bigoplus g)(u) 
			\\&=& \hspace{-3mm} 2\sqrt{\frac{1}{2\pi i b}} \sqrt{\frac{1}{2\pi b i}}\int_{\mathbb{R}}\int_{\mathbb{R}} \big[e^{\frac{i}{2b}\big[a(\tau+\xi)^2+4\tau(u_0-u)+4\xi(u_0-u)-2u(du_0-b\omega_0)+du^2-2a\tau^2-2a\tau\xi-2\tau u_0-2\xi u_0\big]}f(\tau)g(\xi)\big] d\tau d\xi	
\\&=&\hspace{-3mm} 2\sqrt{\frac{1}{2\pi i b}} \sqrt{\frac{1}{2\pi i b}}\int_{\mathbb{R}}\int_{\mathbb{R}} e^{\frac{i}{2b}(a\tau^2+2\tau(u_0-2u)+a\xi^2+2\xi(u_0-2u)-2u(du_0-b\omega_0)+du^2 )}   f(\tau) g(\xi) d\tau d\xi
\\&=& \hspace{-3mm} 2 \sqrt{\frac{1}{2\pi i b}} \sqrt{\frac{1}{2\pi i b}}\int_{\mathbb{R}}\int_{\mathbb{R}} \hspace{-1.5mm}\big[ e^{\frac{i}{2b}(a\tau^2+2\tau(u_0-2u)-4u^2+4du^2+a\xi^2+2\xi(u_0-2u)-4u(du_0-b\omega_0)+4du^2)} \\&&  e^{\frac{-i}{2b}(6u(du_0-b\omega_0)+7du^2)}  f(\tau) g(\xi) \big] d\tau d\xi	
\\&=& \hspace{-3mm} 2 \sqrt{\frac{1}{2\pi i b}} \sqrt{\frac{1}{2\pi i b}}\int_{\mathbb{R}}\int_{\mathbb{R}} \mathcal{K}_{M} f(2\tau)  \mathcal{K}_{M}   g(2\xi) e^{\frac{-i}{2b}(6u(du_0-b\omega_0))}f(\tau) g(\xi) d\tau d\xi	
			\\&=& \hspace{-3mm} 2  e^{\frac{-i}{2b}(6u(du_0-b\omega_0)+7du^2)} \mathcal{O}^{M} f(2u) \mathcal{O}^{M} g(2u).	
		\end{eqnarray*}
		Hence the theorem is proved.
	\end{proof}

	\subsection{OLCT correlation operation}
	\begin{definition}\label{def3}
		Let the functions $f(x), g(x) \in L^1({\mathbb{R}})$. The correlation operation for OLCT is defined as follows
		\begin{eqnarray}
			(f \bigotimes g)(x) 
			\hspace{-3mm}&=&\hspace{-3mm}\sqrt{\frac{i}{2\pi b}}C \int_{\mathbb{R}} f(\tau) 	g(\frac{x}{2}-\tau) e^{\frac{i}{2b}[d(\frac{3x^2}{4}+\tau x-2\tau^2)-4(du_0-b\omega_0)]} d\tau.
		\end{eqnarray}
		where
		\begin{eqnarray}
			C= e^{\frac{i}{2}\big[(cdu_0^2-2adu_0\omega_0+ab\omega_0^2)\big]}
		\end{eqnarray}
	\end{definition}
	
	\begin{theorem}
		Let the functions $f(x), g(x) \in L^2(\mathbb{R})$. $F_M(u)=\mathcal{O}^{M} f(u)$ and $G_M(u)=\mathcal{O}^{M} g(u)$ denote the OLCT of $f(x)$ and $g(x)$ respectively. Then 
		\begin{eqnarray}
			\mathcal{O}^{M} \big[\mathcal{T}(x) f(2x)g(2x)\big](\omega)= F_M\bigotimes G_M(u).
		\end{eqnarray}
		
	\end{theorem}
	\begin{proof}
		By using the definition \ref{def3}, 
		\begin{eqnarray*}
			&&F_M\bigotimes G_M(u)\\&=&\hspace{-3mm}
			\sqrt{\frac{i}{2\pi b}}C \int_{\mathbb{R}} F_M(v) G_M(\frac{u}{2}-v) e^{\frac{i}{2b}[d(\frac{3u^2}{4}+uv-2v^2)-u(du_0-b\omega_0)]} dv		
			\\&=& \hspace{-3mm}\sqrt{\frac{1}{2\pi i b}} \sqrt{\frac{i}{2\pi b}}\int_{\mathbb{R}} \int_{\mathbb{R}}\hspace{-1.5mm} e^{\frac{i}{2b}(ax^2+2x(u_0-u)-2v(du_0-b\omega_0)-u(du_0-b\omega_0)+d(\frac{3u^2}{4}+v^2+uv-2v^2)}f(x) G_M(\frac{u}{2}-v) dv dx.
		\end{eqnarray*}
		Changing the variable $\frac{u}{2}-v = \omega$, we get
		\begin{eqnarray*}
			&&	F_M\bigotimes G_M(u)
			\\&=&\hspace{-3mm}\sqrt{\frac{1}{2\pi i b}}\sqrt{\frac{i}{2\pi b}}\int_{\mathbb{R}} \int_{\mathbb{R}} \big[e^{\frac{i}{2b}\{ax^2+ 2x(u_0-\frac{u}{2})+2x\omega -2u(du_0-b\omega_0)+2\omega(du_0-b\omega_0)+d(u^2-\omega^2)\}}f(x) G_M(\omega) \big] dx d\omega
			\\&=&\hspace{-3mm}\sqrt{\frac{1}{2\pi i b }}\sqrt{\frac{i}{2\pi b}}\int_{\mathbb{R}} \int_{\mathbb{R}}\hspace{-1mm} e^{\frac{-i}{2b}(ax^2+ 2x(u_0-\omega)-2\omega (du_0-b\omega_0))}G_M(\omega) e^{\frac{i}{2b}(2ax^2+2x(u_0-\frac{u}{2})+2x(u_0-\omega)+2x\omega -2u(du_0-b\omega_0)+du^2)}f(x)  dx d\omega
			\\&=&\hspace{-3mm}\sqrt{\frac{i}{2\pi i b}} \int_{\mathbb{R}} g(x)e^{\frac{i}{2b}(2ax^2+x(4u_0-u)-2u(du_0-b\omega_0)+du^2)}f(x)  dx
			\\&=&\hspace{-3mm}\sqrt{\frac{i}{2\pi i b}} \int_{\mathbb{R}} g(x)e^{\frac{i}{2b}(a(\frac{x}{2})^2+2\frac{x}{2}(u_0-u)-2u(du_0-b\omega_0)+du^2)}e^{\frac{i}{2b}(7a\frac{x^2}{4}+6\frac{x}{2}u_0)}f(x)  dx.
		\end{eqnarray*}
		Changing the variable $\frac{x}{2}$ to $x$, we get
		\begin{eqnarray*}
			(F_M \bigotimes G_M)(u)&=&\sqrt{\frac{1}{2\pi i b}} \int_{\mathbb{R}}  e^{\frac{i}{2b}(ax^2+2x(u_0-u)-2u(du_0-b\omega_0)+du^2)}e^{\frac{i}{2b}(7ax^2+6xu_0)}f(2x)g(2x) 2dx
			\\&=& \mathcal{O}^M \big[2e^{\frac{i}{2b}(7ax^2+6u_0x)} f(2x)g(2x)\big](u).
		\end{eqnarray*}
		This completes the theorem.
	\end{proof}
	
	The convolution and correlation are mathematical operations that applies on two functions to produce a third function. 
	\section{Spectral related theorems}	
	In this Section, we discuss analogous type theorems of real PW and Boas theorem for OLCT. To prove the real PW theorem, we will use the following lemma.
	
	\begin{lemma}\label{lem2}
		Let $f \in L^1 (\mathbb{R})$ be infinitely differentiable and $\mathcal{O}^Mf$ be its OLCT. Then 
		\begin{eqnarray}\label{4.1}
			(\mathcal{O}^M \Delta_x ^n f)(u) = (\frac{-i}{b}u)^n (\mathcal{O}^M f)(u),
		\end{eqnarray}
		where
		\begin{eqnarray}
			\Delta_x= -\big( \frac{\partial{}}{\partial{x}} + \frac{i}{b}(ax+u_0) \big)
		\end{eqnarray}
		and
		\begin{eqnarray}
			\Delta_x^n=(-1)^n \sum_{m=0}^{n} P_m (x) D^{n-m}_x, 
		\end{eqnarray}
		where
		\begin{eqnarray*}
			P_m(x)=\binom nx (\frac{i}{b})^m \sum_{k=0}^{m} \binom mk (ax)^k u_0^{m-k}. 
		\end{eqnarray*}
		Here $P_m$ denotes the $m^{th}$ order polynomial with $P_0 (x)=1$.
	\end{lemma}	
	
	\begin{proof}
		Let $\Delta_x^* = \frac{\partial{}}{\partial{x}} - i(ax +u_0)$.
		Then the kernel of OLCT satisfies the following
		\begin{eqnarray}
			(\Delta_x^*)\mathcal{K}_{M} f(u,x)= (\frac{-i}{b}u) \mathcal{K}_{M} (u,x). \label{4.4}
		\end{eqnarray}
		For $n \in \mathbb{N}$,
		\begin{eqnarray}
			(\Delta_x^*)^n\mathcal{O}^{M}(u,x)&=& (\frac{-i}{b}u)^n \mathcal{K}_M(u,x).\label{4.5} \\
			\int_{\mathbb{R}}(\Delta_x^*)^n\mathcal{K}_M(u,x) f(x)dx&=& \int_{\mathbb{R}}\mathcal{K}_M(u,x)(\Delta_x)^n f(x)dx.\label{4.6}
		\end{eqnarray}
		Using \eqref{4.4},\eqref{4.5} and \eqref{4.6}, we get 
		\begin{eqnarray*}
			(\mathcal{O}^M \Delta_x ^n f)(u) &=&
			\int_{\mathbb{R}} \mathcal{K}_M(u,x) \Delta_x^n f(x) dx \\&=& \int_{\mathbb{R}}[(\Delta_x^*)^n \mathcal{K}_M(u,x)] f(x) dx \\&=&
			\int_{\mathbb{R}}[(\frac{-i}{b}u)^n \mathcal{K}_M(u,x)] f(x) dx \\&=& (\frac{-i}{b}u)^n\int_{\mathbb{R}} \mathcal{K}_M(u,x) f(x) dx \\&=&
			(\frac{-i}{b}u)^n(\mathcal{O}^M f)(u).
		\end{eqnarray*}
		Hence we obtain the required result.
	\end{proof}

	\subsection{Real Paley-Weiner Theorem associated with OLCT}
	\begin{theorem}
		Let $f\in L^2 (\mathbb{R})$ be infinitely differentiable. Then $\big[-{\gamma}{|b|} , {\gamma}{|b|}\big]$ is the compact support of $(\mathcal{O}^M_ f) (u)$ if and only if $\Delta_t^n f \in L^2 (\mathbb{R})$ and $(\frac{u}{b})^n (\mathcal{O}^M f)(u) \in L^2 (\mathbb{R})$, for all $n \in \mathbb{Z}_+ $ with
		\begin{eqnarray}
			\lim_{n \to \infty} ||\Delta_t^n f||^{\frac{1}{n}}= \gamma,
		\end{eqnarray}
		where $\gamma= \sup \{|\frac{u}{b}|: \hspace{1mm} b \neq 0, \hspace{1mm} (\mathcal{O}^{M} f)(u) \neq 0, \hspace{1mm} u\in \mathbb{R} \}$.
	\end{theorem}
	\begin{proof}
		(Necessary Condition)
		Suppose that  $\big[-{\gamma}{|b|} , {\gamma}{|b|}\big]$ is the compact support of $(\mathcal{O}^M_ f) (u)$ and  $f\in L^2 (\mathbb{R})$.
		\\By Lemma \ref{nm} and \ref{lem2}, we have
		\begin{eqnarray*}
			||\Delta_x ^n f||_2^2 
			&=& \int_{\mathbb{R}} |(\frac{-i}{b}u)^n (\mathcal{O}^{M} f)(u)|^2 du \\&=& \int_{-{\gamma}{|b|}}^{{\gamma}{|b|}} |(\frac{-i}{b}u)|^{2n} |(\mathcal{O}^{M}f)(u)|^2 du \\ &\leq& \sup_{u\in [{\gamma}{|b|},{\gamma}{|b|}]} |\frac{i}{b}u|^{2n} \int_{{\gamma}{|b|}}^{{\gamma}{|b|}} |(\mathcal{O}^{M}f)(u)|^2 du \\&=& \gamma^{2n} ||\mathcal{O}^{M}f||^2_2,
		\end{eqnarray*}
		which leads to
		\begin{eqnarray*}
			||\Delta_x ^n f||^{\frac{1}{n}} \leq \gamma ~ ||\mathcal{O}^{M }f||^{\frac{1}{n}}.
		\end{eqnarray*}
		Hence
		\begin{eqnarray}
			\lim_{n \to \infty} \sup ||\Delta_x^n f||^{\frac{1}{n}} &\leq& \gamma.\label{rp2}
		\end{eqnarray}
		Now, let $\epsilon \in (0,\frac{\gamma}{2})$. Then
		\begin{eqnarray*}
			||\Delta_x ^n f||^2 &=&\int_{{-\gamma}{|b|}}^{{\gamma}{|b|}} |(\frac{i}{b}u)^{n} (\mathcal{O}^{M}f)(u)|^2 du
			\\&\geq& \int_{{-(\gamma-2\epsilon)}{|b|}}^{{\gamma-\epsilon}{|b|}}|(\frac{i}{b}u)^{n} (\mathcal{O}^{M}f)(\omega)|^2 du
			\\&\geq&  \inf_{u\in [{-(\gamma-2\epsilon)}{|b|},{\gamma-\epsilon}{|b|}]} |\frac{i}{b}u|^{2n} \int_{\frac{-(\gamma-2\epsilon)}{|b|}}^{\frac{\gamma-\epsilon}{|b|}} |(\mathcal{O}^{M}f)(u)|^2 du 
			\\&=& (\gamma-2\epsilon)^{2n}\int_{-{(\gamma-2\epsilon)}{|b|}}^{\gamma-\epsilon{|b|}} |(\mathcal{O}^{M}f)(u)|^2 du \\&\geq& (\gamma-2\epsilon)^{2n}||\mathcal{O}^{M}f||^2_2 .
		\end{eqnarray*}
		i.e.,
		\begin{eqnarray*} 
			\lim_{n \to \infty} \inf ||\Delta_x ^n f||^{\frac{1}{n}} &\geq& (\gamma-2\epsilon), \hspace{2mm}\forall \epsilon \in (0, \frac{\gamma}{2}).
		\end{eqnarray*}
		Since $\epsilon > 0 $ is arbitrary,
		\begin{eqnarray}	
			\lim_{n \to \infty} \inf ||\Delta_x ^n f||^{\frac{1}{n}} &\geq& \gamma. \label{rp3}
		\end{eqnarray}
		Therefore, from \ref{rp2} and \ref{rp3}, we conclude that
		\begin{eqnarray*}
			\lim_{n \to \infty}||\Delta_x ^n f||^{\frac{1}{n}} &=& \gamma.
		\end{eqnarray*}
		\\(Sufficient Condition): Suppose the derivative operator $\Delta_x^n f \in L^2 (\mathbb{R})$ and $(\frac{u}{b})^n (\mathcal{O}^{M} f)(u) \in L^2 (\mathbb{R})$, for all $n \in \mathbb{Z}_+ $ with
		\begin{eqnarray*}
			\lim_{n \to \infty} ||\Delta_x^n f||^{\frac{1}{n}}= \gamma < \infty,
		\end{eqnarray*}
		where 
		\begin{eqnarray*}
			\gamma= \sup \{|\frac{u}{b}|,\hspace{1mm} b \neq 0 ;\hspace{2mm} (\mathcal{O}^{M} f)(u) \neq 0, \hspace{2mm} u \in \mathbb{R} \} < \infty.
		\end{eqnarray*}
		This gives 
		\begin{eqnarray*}
			|u| \leq {\gamma}{|b|}, \hspace{3mm} \forall u\in \mathbb{R} \hspace{3mm} with \hspace{3mm} (\mathcal{O}^{M} f)(u) \neq 0.
		\end{eqnarray*}
		Hence $\big[-{\gamma}{|b|} , {\gamma}{|b|}\big]$ is the compact support of $(\mathcal{O}^M_ f) (u)$.
	\end{proof}
	
	\subsection{Boas theorem associated with OLCT}
	\begin{theorem}
		Let the function $f \in L^1(\mathbb{R})$. Then the OLCT of $f$ vanishes in the neighborhood of the origin if and only if $B^n f \in L^1 (\mathbb{R})$ is well-defined 
		,$\forall n \in \mathbb{Z}_+$ and \begin{eqnarray}
			\lim_{n \to \infty} ||B^nf||^\frac{1}{n}= R < \infty,
		\end{eqnarray}
		where
		\begin{eqnarray*}
			(Bf)(x)&=&e^{\frac{-i}{2b}(ax^2+2xu_0)}\int_{x}^{\infty} e^{\frac{i}{2b} i(at^2+2tu_0)}f(t)dt,
			\\R=\gamma^{-1} \hspace{3mm} and \hspace{3mm}
			\gamma&=&\{\inf{|\frac{1}{b}u|: u\in\sup\mathcal{O}^{M} f}\}.
		\end{eqnarray*}
	\end{theorem}
	
	\begin{proof}(Necessary condition)
		Suppose	$(\mathcal{O}^{M} f)(u)$ vanishes on $({\gamma}{|b|},{\gamma}{|b|})$ such that
		\begin{eqnarray}
			(B_Nf)(x)&=&e^{\frac{-i}{2b}(ax^2+2xu_0)}\int_{x}^{N} e^{\frac{-i}{2b}(at^2+2tu_0)}f(t)dt.
		\end{eqnarray}
		Obviously
		\begin{eqnarray}
			\lim_{N \to \infty} (B^Nf)(x)&=& (Bf)(x). 
		\end{eqnarray}
		The characteristic function  $\chi_N^x$ is defined by
		\begin{eqnarray*}
			\chi_{N}^x(t)=
			\begin{cases}
				1 & , \hspace{3mm} \text{if }  t \in [x,N]\\
				0 & ,\hspace{3mm} \text{otherwise }. 
			\end{cases}
		\end{eqnarray*}
		Now, using \ref{def1} and \ref{nm}, we get
		\begin{eqnarray*}
			&&	(B_Nf)(x)
			\\&=& e^{\frac{-i}{2b}(ax^2+2xu_0)}\big< e^{\frac{i}{2b}(at^2+2tu_0)}f(t), \chi_{N}^x(t) \big>
			\\&=& e^{\frac{-i}{2b}(ax^2+2xu_0)}\big< f(t), e^{\frac{-i}{2b}(at^2+2tu_0)}\chi_{N}^x(t) \big>		
			\\&=& e^{\frac{-i}{2b}(ax^2+2xu_0)}\big< (\mathcal{O}^{M} f)(u), (\mathcal{O}^{M} e^{\frac{-i}{2b}(at^2+2tu_0)}\chi_{N}^x)(u) \big>
			\\&=& e^{\frac{-i}{2b}(ax^2+2xu_0)}\int_{\mathbb{R}} (\mathcal{O}^{M} f)(u) \overline{(\mathcal{O}^{M} e^{-i(at^2+2xu_0)}\chi_{N}^x)(u)} du
			\\&=& e^{\frac{-i}{2b}(ax^2+2xu_0)}\int_{\mathbb{R}} (\mathcal{O}^{M} f)(u) {\big(\overline {\int_{x}^{N} \mathcal{K}_{M}(u,t) e^{\frac{-i}{2b}(at^2+2tu_0)}dt} }\big) du		
			\\&=&\overline{\sqrt{\frac{1}{2\pi i b}}} e^{\frac{-i}{2b}(ax^2+2xu_0)}\int_{\mathbb{R}}\big[ (\mathcal{O}^{M} f)(u)\\&& \times \overline{\big( \int_{x}^{N} e^{\frac{i}{2b}(at^2+2t(u_0-u)-2u(du_0-b\omega_0)+du^2)}e^{\frac{i}{2b}(at^2+2tu_0)}dt \big)\big]} du	
	\\&=&\overline{\sqrt{\frac{1}{2\pi i b }}} e^{\frac{-i}{2b}(ax^2+2xu_0)}\int_{\mathbb{R}} (\mathcal{O}^{M} f)(u) {\big( \int_{x}^{N}}\overline{ e^{\frac{i}{2b}(-2tu-2u(du_0-b\omega_0)+du^2+du_0^2)}dt} \big) du
	\\&=&\overline{\sqrt{\frac{1}{2\pi i b }}} e^{\frac{-i}{2b}(ax^2+2xu_0)}\int_{\mathbb{R}} (\mathcal{O}^{M} f)(u)\overline{ e^{\frac{i}{2b}(-2u(du_0-b\omega_0)+du^2+du_0^2)} \big( \frac{e^{\frac{-i}{2b}(2tu)}-e^{\frac{-i}{2b}(2xu)}}{\frac{-i}{2b}2u} }\big) du	
\\&=&\overline{\sqrt{\frac{1}{2\pi i b}}} e^{\frac{-i}{2b}(ax^2+2xu_0)}\int_{\mathbb{R}}\frac{b}{iu} (\mathcal{O}^{M} f)(u)e^{\frac{-i}{2b}(-2u(du_0-b\omega_0)+du^2+du_0^2)} e^{\frac{-i}{2b}2xu}du	\\&~~~-&\overline{\sqrt{\frac{1}{2\pi i b}}} e^{\frac{-i}{2b}(ax^2+2xu_0)}\int_{\mathbb{R}}\frac{b}{iu} (\mathcal{O}^{M} f)(u)e^{\frac{-i}{2b}(-2u(du_0-b\omega_0)+du^2+du_0^2)} e^{\frac{-i}{2b}2Nu} du	
			\\&=&\overline{\sqrt{\frac{1}{2\pi i b}}} \int_{\mathbb{R}}\frac{b}{iu} (\mathcal{O}^{M} f)(u) e^{\frac{-i}{2b}(ax^2+2xu_0-2u(du_0-b\omega_0)+du^2+du_0^2)}du	\\&~~~-&\overline{\sqrt{\frac{1}{2\pi i b}}} e^{\frac{-i}{2b}(ax^2+2xu_0)}e^{\frac{i}{2b}(aN^2+2Nu_0)}\int_{\mathbb{R}}\frac{b}{iu} (\mathcal{O}^{M} f)(\omega)e^{\frac{-i}{2b}(aN^2 +2N(u_0-u) -2u(du_0-b\omega_0)+du^2+du_0^2)} du
			\\&=& \int_{\mathbb{R}} \frac{b}{iu}(\mathcal{O}^{M} f)(u) \overline {\mathcal{K}_{M }(u,x)} du
			\\&-& e^{\frac{-i}{2b}(ax^2+2xu_0-aN^2-2Nu_0)}\int_{\mathbb{R}} \frac{b}{iu}(\mathcal{O}^{M} f)(u) \overline {\mathcal{K}_{M }(u,N)} du.
		\end{eqnarray*}
		As $N \to \infty$, by lemma \ref{lem1}, we have 
		\begin{eqnarray*}
			\int_{\mathbb{R}} (\mathcal{O}^{M} f)(u) \overline {\mathcal{K}_{M}(u,N)} du &\to& 0.
		\end{eqnarray*}
		Thus, we get 		
		\begin{eqnarray*}
			\lim_{N\to \infty} (B_N f)(x)&=& \int_{\mathbb{R}} \frac{b}{iu}(\mathcal{O}^{M} f)(u) \overline {\mathcal{K}_{M}(u,N)} du \\&=& \mathcal{O}^{M^{-1}} \big(\frac{b}{iu} \mathcal{O}^{M} f\big)(x).
		\end{eqnarray*}
		Therefore
		\begin{eqnarray*}
			(B f)(x)=\mathcal{O}^{M^{-1}} \big(\frac{b}{iu} \mathcal{O}^{M} f\big)(x).
		\end{eqnarray*}	
		Applying OLCT on both sides gives
		\begin{eqnarray}
			\mathcal{O}^{M}(B f)(u)= \frac{b}{iu} \big(\mathcal{O}^{M} f\big)(u),
		\end{eqnarray}	
		which leads to
		\begin{eqnarray*}	
			\mathcal{O}^{M}(B^2 f)(u)= \big(\frac{b}{iu}\big)^2 \big(\mathcal{O}^{M} f\big)(u).
		\end{eqnarray*}
		Using Mathematical Induction, we get
		\begin{eqnarray}\label{4.14}
			\mathcal{O}^{M}(B^n f)(u)= \big(\frac{b}{iu}\big)^n \big(\mathcal{O}^{M} f\big)(u).		
		\end{eqnarray}
		Combining \ref{nm} and \eqref{4.14}, we get
		\begin{eqnarray*}
			||B^n f||^2_2 
			&=& \int_{\mathbb{R}}|(\mathcal{O}^{M} B^n f) (u)|^2 du
			\\&=&\int_{\mathbb{R}}|\big(\frac{b}{u}\big)^n \big(\mathcal{O}^{M} f\big)(u)|^2du	
		\end{eqnarray*}
		Let $E= \mathbb{R} \setminus \big(-{\gamma}{|b|},{\gamma}{|b|} \big)$. Then $\mathcal{O}^{M} f$ vanishes outside of $E$.
		\begin{eqnarray*}
			||B^n f||^2_2 &=& \int_{\mathbb{R}}|\big(\frac{b}{u}\big)^n \big(\mathcal{O}^{M}  f\big)(u)|^2du
			\\&\leq&\frac{1}{\gamma^{2n}}\int_{E}|(\mathcal{O}^{M} f)(u)|^2 du
			\\&=& \frac{1}{\gamma^{2n}}||\mathcal{O}^{M} f||^2,			
		\end{eqnarray*}
		which gives
		\begin{eqnarray}\label{4.15}
			\lim_{n\to \infty} \sup ||B^n f (x)||^{\frac{1}{n}} &\leq& \frac{1}{\gamma} .		
		\end{eqnarray}
		For obtaining the reverse inequality, let $G=\{ u\in \mathbb{R}: \gamma< |\frac{b}{u}|<\gamma+\epsilon\}$. Then
		\begin{eqnarray}\label{4.16}
			||B^n f(x)||^2 &=& \int_{\mathbb{R}}|\big(\frac{b}{u}\big)^n \big(\mathcal{O}^{M}  f\big)(u)|^2du \nonumber
			\\&\geq& \int_{G}|\big(\frac{b}{u}\big)^n \big(\mathcal{O}^{M} f\big)(u)|^2du	\nonumber
			\\ &\geq& (\gamma+\epsilon)^{-2n}\int_{G}|\big(\mathcal{O}^{M}  f\big)(u)|^2du. \nonumber
		\end{eqnarray}	
		Thus, we get
		\begin{eqnarray*}
			\lim_{n\to \infty} \inf||B^n f||^{\frac{1}{n}}&\geq& (\gamma+\epsilon)^{-1},
		\end{eqnarray*}	
		which leads to
		\begin{eqnarray}
			\lim_{n\to \infty} \inf||B^n f||^{\frac{1}{n}}&\geq& \frac{1}{\gamma}.
		\end{eqnarray}
		Combine \ref{4.15} and \ref{4.16} to get
		\begin{eqnarray*}
			\lim_{n\to \infty} ||B^n f||^{\frac{1}{n}}&=& \frac{1}{\gamma}= R.			 
		\end{eqnarray*}
		(Sufficient condition) Suppose $B^n f \in L^2 (\mathbb{R})$ is well-defined, 
		$\forall n$ and
		\begin{eqnarray*}
			\lim_{n\to \infty} ||B^n f||^{\frac{1}{n}}=R <\infty,
		\end{eqnarray*}
		where 
		\begin{eqnarray*}
			(Bf)(x) =e^{\frac{-i}{2b}(ax^2+2xu_0)}\int_{x}^{\infty} e^{\frac{i}{2b}(ax^2+2xu_0)}f(x)dx.
		\end{eqnarray*}
		Suppose $\mathcal{O}^{M} f$ does not vanish in any neighborhood of origin.
		\\Calculating in similar way, we get
		\begin{eqnarray*}
			(\Delta_x)^n B^n f (x) = f(x).
		\end{eqnarray*}
		Applying OLCT on both sides, we obtain
		\begin{eqnarray}\label{4.17}
			\mathcal{O}^{M}(\Delta_x^n B^n f(x))(u)=(\mathcal{O}^{M} f)(u).
		\end{eqnarray}
		Replacing $f$ by $B^n f$ in equation \eqref{4.1} gives
		\begin{eqnarray}\label{4.18}
			\mathcal{O}^{M}(\Delta_x^n B^n f(x))(u)&=& (\frac{-i}{b}u)^n(\mathcal{O}^{M} B^n f(x))(u).
		\end{eqnarray}	
		Combining \eqref{4.17} and \eqref{4.18} leads to 
		\begin{eqnarray*}
			(\frac{-i}{b}u)^n(\mathcal{O}^{M} B^n f(x))(u)&=&(\mathcal{O}^{M} f)(u).
		\end{eqnarray*}	
		Hence
		\begin{eqnarray}\label{4.19}
			(\mathcal{O}^{M} B^n f(x))(u)&=& \frac{(\mathcal{O}^{M} f)(u)}{(\frac{-i}{b}u)^n} .
		\end{eqnarray}
		Let $G_0 = \{ u \in \mathbb{R}: 0 < |\frac{u}{b}| <\epsilon \}$. Then using \ref{nm} and \eqref{4.19}, we get
		\begin{eqnarray*}
			||B^n f||^2 
			&=&\int_{\mathbb{R}}\frac{|\mathcal{O}^{M} f(u)|^2}{|\frac{-i}{b}u|^{2n}} du
			\\&\geq& \int_{G_0}\frac{|\mathcal{O}^{M} f(u)|^2}{|\frac{u}{b}|^{2n}} du
			\\&\geq& \frac{1}{\epsilon^{2n}} \int_{G_0} |\mathcal{O}^{M} f(u)|^2 du,
		\end{eqnarray*}
		which gives
		\begin{eqnarray*}
			||B^n f||^{\frac{1}{n}} &\geq& \frac{1}{\epsilon} \big( \int_{G_0} |\mathcal{O}^{M} f(u)|^2 du \big)^{\frac{1}{2n}}.
		\end{eqnarray*}	
		Thus, we get
		\begin{eqnarray*}
			\lim_{n \to \infty}||B^n f||^{\frac{1}{n}} &\geq& \frac{1}{\epsilon}.
		\end{eqnarray*}
		As $\epsilon \to 0$, we obtain
		\begin{eqnarray*}
			\lim_{n\to \infty}||B^n f||^{\frac{1}{n}} = \infty.
		\end{eqnarray*}
		This contradicts our assumption.
		Hence OLCT $(\mathcal{O}^{M} f)$ of $f$ vanishes in the neighborhood of origin.
	\end{proof}


	\section{Potential Applications}
	In this section, we discuss filter design based on convolution and product structures associated with the OLCT. Several studies have explored the application of multiplicative filters designed using the FrFT or LCT to eliminate noise or distortion. Using convolution, we examine the design approaches and performance of multiplicative filters created with the OLCT. Additionally, we demonstrate that filtering is more easily obtained by convolution in the time domain. A filtering model is presented, enabling the design of low-pass, band-pass, and high-pass multiplicative filters. The impact of the multiplicative filter is expressed by the following equation.
	\begin{eqnarray*}
		r_{out}(x)
		&=&\mathcal{O}^{M^{-1}}\big[\mathcal{O}^M (f_{in}\bigoplus g)\big]
	\\&=&\mathcal{O}^{M^{-1}}\big[\mathcal{O}^M(f_{in})2\mathcal{T}(u){G}_M(2u)\big](x) 
	\end{eqnarray*} 
$where\hspace{3mm}  \tilde{G}_M(2u)=2\mathcal{T}(u){G}_M(2u)$.

$\mathcal{O}^M(f_{in})$ represents the OLCT of the received signal $f_{in}(t)$. The filtering model is presented, and we can design low-pass, band-pass, and high-pass multiplicative filter models using $\tilde{G}_A(2u)$.
 
 In practical applications, when the received signal $f_{in}(t)$ consists of two components the desired signal $f(t)$ and noise $n(t)$ a multiplicative filter in the OLCT domain can effectively preserve the desired signal and significantly reduce the noise, thereby enhancing the signal-to-noise ratio (SNR).    
 
In \cite{pw9,pw10}the Paley–Wiener theorem as a characterization that relates the support of a function to its growth, providing an understanding of the image of a space of functions or distributions under a Fourier-type transform. In this context, we aim to extend the classical Fourier transform concept into the OLCT domain, as OLCT serves as a generalized version of the FT and is a valuable tool in optics and signal processing. This method can be visualized in figure \ref{fig1}.
	\begin{figure}[t!]
		\centering
		\begin{tikzpicture}[node distance=2cm]
			\tiny	
			\hspace{-8mm}\tikzstyle{input} = [minimum width=0.6cm, minimum height=0.75cm, text centered, text width=0.6cm, fill=white]
			\tikzstyle{process} = [draw, rectangle, minimum width=3cm, minimum height=0.75cm, text centered, text width=2.8cm, fill=white]
			\tikzstyle{modulator} = [minimum size=0.05cm, text centered, text width=0.05cm, fill=white]
			\tikzstyle{output} = [draw, rectangle, minimum width=6cm, minimum height=0.75cm, text centered, text width=7.5cm, fill=white]
			
			\node (input1) [input] {\tiny$f_{in}(x)$};
			\node (process1) [process, right of=input1, xshift=0.5cm] {\tiny$\mathcal{O}^M (f_{in})(u),$\\ $u \in [u_1,u_2]$};
			\node (process2) [process, right of=process1, xshift=1.5cm] {\tiny$\mathcal{O}^M (f_{in})(2u),$\\$ u \in [\frac{u_1}{2},\frac{u_2}{2}]$};
			\node (modulator) [modulator, right of=process2, xshift=0.5cm] {$\hspace{-1mm}\bigoplus$};
			\node (output) [output, below of=process2, xshift=2.5cm] {\tiny{$\mathcal{O}^{M^{-1}}\big[\mathcal{O}^M(f_{in})2\mathcal{T}(u){G}_M(2u)\big](x)$}};
			\node (process3) [process, right of=modulator, xshift=0.3cm] {\tiny$e^{\frac{-i}{2b}(6u(du_0-b\omega_0)+7du^2)}{G}(2u)$};
			\node (additionalOutput) [right of=output, xshift=-2cm, yshift=-1.5cm] {\tiny$f_{out}(x)$};

			\draw [->] (input1) -- (process1);
			\draw [->] (process1) -- (process2);
			\draw [->] (process2) -- (modulator);
			\draw [->] (modulator) -- (output);
			\draw [->] (output) -- (additionalOutput);
			\draw [->, shorten >= 1mm] (process3) -- (modulator);
		\end{tikzpicture}
		\caption{}
		\label{fig1}
	\end{figure}


	\section{Conclusion}
In this article, we introduce a new type of convolution and correlation theorems associated with the OLCT. Additionally, we discuss their applications in multiplicative filter design, which may prove beneficial in optics and signal processing for signal recovery. Furthermore, we explore the real Paley-Wiener (PW) and Boas theorems for the OLCT, analyzing the characteristics of signals in the OLCT within the $L^2$ domain. These studies can be further extended to $L^p$ spaces and Schwartz-type function spaces.
	
	\section*{Funding statement} 
	This research received no specific grant from any funding agency in the public, commercial, or not-for-profit sectors.
	\section*{Declaration of competing interest}
	The authors declare that they have no known competing financial interests or personal relationships that could have appeared to influence the work reported in this paper.
	\\\\

	\bibliographystyle{amsplain}

\end{document}